\documentclass[a4paper,11pt, reqno]{amsart}
\usepackage{amsfonts,amssymb,amsmath,amsthm}
\usepackage{url}
\usepackage{enumerate}
\usepackage{bbm}
\usepackage{amssymb}
\usepackage{color}

\newtheorem{theorem}{Theorem}[section]
\newtheorem{lemma}[theorem]{Lemma}
\newtheorem{proposition}[theorem]{Proposition}

\theoremstyle{definition}

\newtheorem{remark}[theorem]{Remark}
\numberwithin{equation}{section}

\DeclareMathOperator{\conv}{conv}

\newcommand{\N}{\ensuremath{{\mathbb N}}}

\newcommand{\B}{\ensuremath{{\mathbb B}}}

\newcommand{\K}{\ensuremath{{\mathbb K}}}

\newcommand{\SSS}{\ensuremath{{\mathbb S}}}
\newcommand{\R}{\ensuremath{{\mathbb R}}}

\newcommand{\vol}{\mathrm{vol}}

\newcommand\bd{{\rm bd}}

\def\dint{\textup{d}}
\newcommand{\bM}{\ensuremath{{\mathbf m}}}
\newcommand{\bP}{\ensuremath{{\mathbf P}}}
\newcommand{\bE}{\ensuremath{{\mathbf E}}}

\author[J. H\"orrmann]{Julia H\"orrmann}
\address{Julia H\"orrmann: Faculty of Mathematics, Ruhr University Bochum, Universit\"atsstra\ss e 150, 44780 Bochum, Germany}
\email{julia.hoerrmann@rub.de}

\author[J. Prochno]{Joscha Prochno}
\address{Joscha Prochno: School of Mathematics \& Physical Sciences, University of Hull, Cottingham Road,
	Hull, HU6 7RX, United Kingdom}
\email{j.prochno@hull.ac.uk}

\author[C. Th\"ale]{Christoph Th\"ale}
\address{Christoph Th\"ale: Faculty of Mathematics, Ruhr University Bochum, Universit\"atsstra\ss e 150, 44780 Bochum, Germany}
\email{christoph.thaele@rub.de}


\begin{document}

\title[Isotropic constant of random polytopes]{On the isotropic constant of random polytopes with vertices on an $\ell_p$-sphere}

\begin{abstract}
The symmetric convex hull of random points that are independent and distributed according to the cone probability measure on the $\ell_p$-unit sphere of $\R^n$ for some $1\leq p < \infty$ is considered. We prove that these random polytopes have uniformly absolutely bounded isotropic constants with overwhelming probability. This generalizes the result for the Euclidean sphere ($p=2$) obtained by D.\ Alonso-Guti\'errez. The proof requires several different tools including a probabilistic representation of the cone measure due to G.\ Schechtman and J.\ Zinn and moment estimates for sums of independent random variables with log-concave tails originating in a paper of E. Gluskin and S. Kwapie\'n.
\end{abstract}
\maketitle


\section{Introduction}


In asymptotic geometric analysis, the hyperplane conjecture is one of the outstanding open problems that first appeared explicitly in a work of Bourgain \cite{Bou}. It asks about the existence of an absolute constant $c>0$ such that every convex body of unit volume has a hyperplanar section of volume bounded from below by $c$, independently of the space dimension. An equivalent formulation, following from a result of Hensley \cite{He}, is that the isotropic constant of every convex body is bounded from above by an absolute constant (a formal definition is provided in Section \ref{sec:prelim} below). Although the hyperplane conjecture has an affirmative answer for several classes of convex bodies such as unconditional convex bodies \cite{Bou,MP}, zonoids, duals of zonoids \cite{B}, bodies with a bounded outer volume ratio \cite{MP}, or unit balls of Schatten norms \cite{KMP}, the general case still remains one of the central open problems in this area. The best general upper bound for the isotropic constant known up to now is due to Klartag \cite{K06} and gives an upper bound of order $n^{1/4}$ with $n$ being the space dimension. This improves by a logarithmic factor the previous bound of Bourgain \cite{B91}. 

Milman and Pajor \cite{MP} discovered an interesting connection between the hyperplane conjecture and geometric properties associated with random polytopes. More precisely, they proved that the second moment of the volume of a random simplex in an isotropic convex body is closely related to the value of its isotropic constant, see \cite[Theorem 3.5.7]{IsotropicConvexBodies}. Furthermore, since the pioneering work of Gluskin \cite{G}, random polytopes are a major source for extremizers and counterexamples in high-dimensional convex geometry. As a consequence, they are natural candidates for a potential counterexample to the hyperplane conjecture stated above. Following this philosophy, the isotropic constant of several classes of random polytopes has been studied in recent years. More exactly, it has been shown in \cite{A,DGG,HHRT,KK} that the isotropic constant of these random polytopes is bounded by an absolute constant with probability tending to one, as the space dimension tends to infinity. The models studied so far are Gaussian random polytopes \cite{KK}, random convex hulls of points from the Euclidean unit sphere \cite{A}, random polytopes that arise from uniform random points chosen in the interior of an isotropic convex body \cite{DGG} and random polyhedra associated with a parametric class of Poisson hyperplane tessellations \cite{HHRT}. Following along the lines of \cite{KK}, it was recently proved independently in \cite{ALT} and \cite{GHT} that if $K_N$ is the symmetric convex hull of $N\geq n$ independent random vectors uniformly distributed in the interior of an $n$-dimensional isotropic convex body $K$, then the isotropic constant $L_{K_N}$ of $K_N$ is bounded by a constant multiple of $\sqrt{\log (2N/n)}$ with overwhelming probability (see also \cite{A11,ABBW,J94} for earlier results). 

The principal aim of the present paper is to investigate the isotropic constant of random polytopes that arise by taking the convex hull of random points chosen with respect to the cone measure from the boundary of an $\ell_p$-sphere in $\R^n$ with $1\leq p<\infty$. With the exception of \cite{A}, where the special case $p=2$ of the Euclidean sphere has been investigated (in that case, cone and surface measure coincide), the isotropic constant of such a class of random polytopes has not been studied yet to the best of our knowledge. On the other hand, random polytopes with vertices on the boundary of a prescribed convex body have previously been investigated from a different angle by B{\"o}r{\"o}czky, Fodor and Hug \cite{BoeroeEtAl}, Reitzner \cite{Reitzner2002,Reitzner2003}, Richardson, Vu and Wu \cite{RichardsonVuWu} as well as Sch\"utt and Werner \cite{SchuettWernerRPBoundary}. These papers concentrate on expectation and variance asymptotic for several key geometric functionals, such as the volume and the vertex number of these random polytopes, as the number of vertices tends to infinity. In contrast, we essentially analyse the geometric behaviour of random polytopes with vertices on the boundary in high dimensions, that is, as the number of vertices and the space dimension tend to infinity simultaneously. 

As in most of the previous results that verified the boundedness of the isotropic constant of random polytopes, our proof follows the approach invented by Klartag and Kozma \cite{KK} (see also \cite{A,DGG} and we refer to \cite{HHRT} for a notable exception based on \cite{HH14}). On the other hand, we would like to emphasize that in the case we study several additional tools are required. These include 

\renewcommand\labelitemi{\tiny$\bullet$}
\begin{itemize}
\item a coupling argument to estimate the volume radius from below,
\item a probabilistic representation of the cone measure of an $\ell_p$-sphere due to Schechtman and Zinn \cite{SZ} (see also \cite{RR91}),
\item moment estimates for sums of independent random variables with log-concave tails proved by Barthe, Gu\'edon, Mendelson and Naor \cite{BGMN}, which are based on earlier work of Gluskin and Kwapie\'n \cite{GK}, and
\item a $\psi_2$-estimate for the cone measure on an $\ell_p$-sphere in the flavour of Bobkov and Nazarov (see \cite{BobkovNazarov}).
\end{itemize}

\medskip

The rest of this paper is structured as follows. Our main result, Theorem \ref{thm:Main} below, together with its mathematical framework, is presented in the next section. In Section \ref{sec:prelim} we collect some preliminary material that is needed in our further arguments. All proofs are presented in Section \ref{sec:ProofThm} at the end of this paper.

\section{Presentation of the result}

By $\K^n$ we denote the class of convex bodies in $\R^n$, that is, the class of compact convex subsets of $\R^n$ having non-empty interior. We denote the $n$-dimensional Lebesgue measure of $K\in\K^n$ by $\vol_n(K)$.

One says that $K\in\K^n$ is isotropic or in isotropic position, if $\vol_n(K)=1$, $K$ has its barycentre at the origin and if
\[
\int_K\langle x,\theta\rangle^2 \, \dint x=L_K^2\,,\qquad \theta\in \SSS^{n-1}\,.
\]
Here, $\langle \,\cdot\,,\,\cdot\,\rangle$ stands for the standard scalar product on $\R^n$ and $L_K$ is a constant independent of $\theta$, the so-called {isotropic constant} of $K$. 

For $1\leq p <\infty$, we denote by $\ell_p^n$ the space $\R^n$ supplied with the norm 
\[
\|x\|_p=\bigg(\sum_{i=1}^n |x_i|^p\bigg)^{1/p},\qquad x=(x_1,\ldots,x_n)\in\R^n.
\]
The unit ball in $\ell^n_p$ is denoted by $\B_p^n=\{x\in\R^n \,:\, \|x\|_p \leq 1 \}$ and we write $\SSS_p^{n-1}=\{x\in\R^n \,:\, \|x\|_p=1 \}$ for the corresponding unit sphere. If $p=2$, we just write $\SSS^{n-1}$ instead of $\SSS_2^{n-1}$.

The {cone (probability) measure} $\bM_K$ of a convex body $K\in\K^n$ is defined as
$$
\bM_K(B) = \frac{\vol_n\big(\{rx:x\in B\,,0\leq r\leq 1\}\big)}{\vol_n(K)}\,,
$$
where $B\subseteq\bd\,K$ is a Borel subset of the boundary $\bd\,K$ of $K$. We remark that the cone measure $\bM_{\B_p^n}$ of $\B_p^n$ coincides with the normalized surface measure on $\SSS_p^{n-1}$ if and only if $p=1$ or $p=2$ (and additionally if $p=\infty$, but we exclude this case in our text). For a more detailed account to the relationship between the cone and the surface measure on $\ell_p$-balls we refer the reader to \cite{N,NR2002}.

Given an isotropic convex body $K\in\K^n$ and $N\geq n+1$, we let $X_1,\ldots,X_N$ be independent random vectors that are distributed on the boundary $\bd\,K$ of $K$ according to the probability law $\bM_K$. By
$$
K_N=\conv\big(\{\pm X_1,\ldots,\pm X_N\}\big)
$$
we denote the symmetric convex hull of these points. The random polytopes $K_N$ are in the focus of our attention in this paper. Our main theorem shows that for $K=\B_p^n$, the $\ell_p$-ball in $\R^n$, the isotropic constant of $K_N$ is absolutely bounded with overwhelming probability. (To simplify the presentation we shall retain the general notation $K_N$ also for the particular choice $K=\B_p^n$, the meaning will always be clear from the context.) 

\begin{theorem}\label{thm:Main}
There exist absolute constants $c,c_1,c_2\in(0,\infty)$ such that if $1\leq p<\infty$, $n+1\leq N\leq e^{\sqrt{n}}$, and $X_1,\dots,X_N\in\SSS_p^{n-1}$ are independent random vectors with distribution $\bM_{\B_p^n}$, then
\[
\bP\big(L_{K_N} \leq c\big) \geq 1-e^{-c_1\sqrt{N}}-2e^{-c_2n\log(2N/n)}\,.
\]
\end{theorem}

The result of Theorem \ref{thm:Main} in particular implies that the probability for the event that the isotropic constant of random polytopes with vertices on any $\ell_p$-sphere $\SSS_p^{n-1}$ stays absolutely bounded (independently of $p$) tends to one exponentially fast, as the space dimension (and simultaneously the number of vertices) tends to infinity. 

\begin{remark}
We remark that for $p=2$, the result reduces to the main theorem of \cite{A}. In fact, in \cite{A} the number $N$ is allowed to vary in the range $N\geq n+1$. It might also be possible to extend Theorem \ref{thm:Main} to this range, but we have made no attempt in this direction. We also remark that if the number of points $N$ is proportional to the space dimension $n$, the main result from \cite{ABBW} already implies boundedness of the isotropic constant $L_{K_N}$ with probability one. 
\end{remark}

In principle, our proof follows along the lines of \cite{KK} (see also \cite{A,DGG}), but requires several additional tools that have not found attention in this context before. On the one hand, we need a high probability lower bound on the volume radius of our random polytope. This is achieved by a coupling argument that allows us to use the known results from \cite{DGT1} when the random points are uniformly distributed inside an isotropic convex body. On the other hand, we need to apply the $\psi_2$-version of Bernstein's inequality and to this end, we need to study the $\psi_2$-behaviour of $\langle X,\theta \rangle$, where $X$ is distributed according to the cone measure of $\B_p^n$ and $\theta\in \SSS^{n-1}$. Since $\B_p^n$ is a $\psi_2$-body iff $p\geq 2$, the $\psi_2$-constant of $n^{1/p}\langle X,\theta \rangle$ is uniformly bounded in that case and we can conclude along the lines of \cite{A,KK}. In contrast, if $1\leq p<2$, we will have to take a different route, which first consists in establishing in this special situation an analogue for the cone measure of a result of Bobkov and Nazarov \cite{BobkovNazarov}. To obtain such an estimate we will carefully combine a probabilistic representation of the cone measure of $\B_p^n$ due to Schechtman and Zinn (see \cite{SZ} as well as \cite{RR91} and also \cite{BGMN,SZ2} for extensions) with moment estimates for sums of independent random variables having logarithmically concave tails due to Gluskin and Kwapie\'n (see \cite{GK} and also \cite{BGMN}). We then build on the methods from a paper of Dafnis, Giannopoulos and Gu\'edon \cite{DGG} to prove boundedness of the isotropic constant in the regime $1\leq p<2$. 

\section{Some Preliminaries} \label{sec:prelim}

\subsection{General notation}

Fix a space dimension $n\geq 1$ and, as in the previous section, denote by $\K^n$ the space of convex bodies in $\R^n$. A convex body $K\in\K^n$ is symmetric (with respect to the origin) if $x\in K$ implies that also $-x$ belongs to $K$.
The {support function} of a convex body $K\in\K^n$ is defined as
$$
h_K(u) = \max\big\{\langle x,u\rangle:x\in K\big\}\,,\qquad u\in\R^n\,.
$$
It is well known that a convex body $K$ is uniquely determined by its support function $h_K$.

\smallskip

For general $K\in\K^n$ we recall that the {isotropic constant} $L_K$ of $K$ satisfies
\begin{equation}\label{eq:IsotropicConstantDef}
nL_K^2 = \min \bigg\{ \frac{1}{\vol_n(AK)^{1+{2\over n}}}\int_{z+AK} \|x\|_2^2\,\dint x \,:\, z\in\R^n, \, A\in {\rm GL}(n)  \bigg\}\,,
\end{equation}
where ${\rm GL}(n)$ is the group of invertible linear transformations of $\R^n$ (see, e.g., \cite[Definition 10.1.6]{AsymGeomAnalyBook}). Moreover, according to \cite[Lemma 11.5.2]{IsotropicConvexBodies} there exists an absolute constant $c\in(0,\infty)$ such that
\begin{equation}\label{eq:IsotropicConstantL1Bound}
nL_K \leq {c\over\vol_n(K)^{1+{1/ n}}}\int_K\|x\|_1\,\dint x
\end{equation}
for any symmetric convex body $K\in\K^n$.

\smallskip

For a symmetric $K\in\K^n$, its (Minkowski) gauge, $\|\cdot\|_K$, is given by 
\[
\|x\|_K = \inf\{\lambda>0 \,:\, x\in \lambda K \}\,,\qquad x\in\R^n\,.
\] 
Obviously this is a norm on $\R^n$, which has $K$ as its unit ball. 
We define the {Minkowski map} $\mu_K:K\to\bd\,K$ by
\[
x\mapsto \mu_K(x)={x\over\|x\|_K}\,.
\]

\smallskip

Finally, for two sequences $(a(n))_{n\in\N}$ and $(b(n))_{n\in\N}$ of real numbers, we write $a(n)\gtrsim b(n)$ (or $a(n)\lesssim b(n)$) provided that there is a constant $c\in(0,\infty)$ such that $a(n)\geq cb(n)$ (or $a(n)\leq cb(n)$) for all $n\in\N$. Moreover, we write $a(n)\approx b(n)$ if $a(n) \lesssim b(n)$ and $a(n) \gtrsim b(n)$. 

\subsection{$\psi_2$-estimates and Bernstein's inequality}

Let $(\Omega,\mathcal F, \bP)$ be a probability space and recall that a convex function $M:[0,\infty)\to[0,\infty)$ with $M(0)=0$ is called an {Orlicz function}. The set of all (equivalence classes of) measurable functions $f:\Omega\to\R$ such that for some $\lambda>0$,
\[
\int_{\Omega}M\left({|f|\over \lambda}\right)\,\dint \bP < \infty,
\]
is called {Orlicz space} associated with $M$ and is denoted by $L_M(\Omega,\bP)$. This space becomes a Banach space when it is supplied with the Luxemburg norm
\[
\|f\|_{M} = \inf\bigg\{ \lambda>0 \,:\, \int_{\Omega}M\left({|f|\over \lambda}\right)\,\dint \bP \leq 1 \bigg\}\,,
\]
which is equivalent to the Orlicz norm on that space.
As already discussed in the previous two sections, we are interested in the case in which the Orlicz function is $\psi_2(x)= e^{x^2}-1$ and in which $\Omega=\SSS_p^{n-1}$ for some $1\leq p<\infty$, $\mathcal F$ is the Borel $\sigma$-field on $\SSS_p^{n-1}$ and $\bP=\bM_{\B_p^n}$ is the cone probability measure of $\B_p^n$. 

We will use the following equivalent expression for the $\psi_2$-norm (see, e.g., \cite[Lemma 2.4.2]{IsotropicConvexBodies}). To formulate it and to simplify notation, instead of $L_M(\Omega)$ for $M(x)=x^q$, $q\geq 1$, let us write $L_q(\Omega)$, and $\|\,\cdot\,\|_{L^q(\Omega)}$ for the corresponding norm.

\begin{proposition}\label{prop:psi 2 norm equivalence}
Let $(\Omega,\mathcal F, \bP)$ be a probability space and $f:\Omega\to\R$ be measurable. Then,
\[
\|f\|_{\psi_2}\approx \sup_{q\geq 2} \frac{\|f\|_{L_q(\Omega)}}{\sqrt{q}}\,.
\]
\end{proposition}

The latter will be used for the space $(\SSS_p^{n-1},\bM_{\B_p^n})$ and for the functionals $\langle\,\cdot\,,\theta \rangle:\SSS_p^{n-1}\to\R$ with $\theta\in\R^n$.

One of the main tools in our proof of Theorem \ref{thm:Main} will be the following Bernstein type inequality obtained in \cite[Proposition 1]{BLM88} (see also \cite[Lemma 11.4.6]{IsotropicConvexBodies}).

\begin{proposition}\label{prop:bernstein}
Let $\xi_1,\dots,\xi_N \in L_{\psi_2}(\Omega,\bP)$ be independent, mean zero random variables defined on a common probability space $(\Omega,\mathcal{F},\bP)$. Assume that $\|\xi_i\|_{\psi_2} \leq c$ for all $i=1,\dots,N$ and some $c\in(0,\infty)$. Then, for all $\varepsilon>0$,
\[
\bP \bigg( \Big| \sum_{i=1}^N \xi_i\Big| > \varepsilon N \bigg) \leq 2\exp\Big({-\frac{\varepsilon^2N}{8c^2}}\Big)\,.
\]
\end{proposition}

\subsection{$L_q$-centroid bodies}

Let $K\in\K^n$ be such that $\vol_n(K)=1$. For $q\geq 1$ the $L_q$-centroid body $Z_q(K)$ is the unique convex body that has support function given by
$$
h_{Z_q(K)}(u) = \bigg(\int_K|\langle x,u\rangle|^q\,\dint x\bigg)^{1/q}\,,\qquad u\in\R^n\,.
$$
We remark that, using the language of centroid bodies, the condition that a convex body $K\in\K^n$ is isotropic can be rephrased by saying that $Z_2(K)$ is a Euclidean ball. The study of $L_q$-centroid bodies, considered from an asymptotic point of view, was initiated by Paouris in \cite{Paouris2,Paouris}.

We recall from the work of Paouris \cite{P06} and of Klartag and Milman \cite{KM12} that the volume radius $\vol_n(Z_q(K))^{1/n}$ of the $L_q$-centroid body of $K$ satisfies the estimate
\begin{equation}\label{eq:volume centroid body}
\vol_n\big(Z_q(K)\big)^{1/n}\approx \sqrt{\frac{q}{n}}\,L_K\,,
\end{equation}
as long as $1\leq q\leq \sqrt{n}$ (compare with \cite[Theorem 2.2]{GHT}). Note that we are working with isotropic convex bodies and not with isotropic log-concave measures, whence the constant $L_K$ in Relation \eqref{eq:volume centroid body}. We remark that this improves upon the previous bound of Lutwak, Yang and Zhang (see \cite[Proposition 5.1.16]{AsymGeomAnalyBook}).  We will see below that \eqref{eq:volume centroid body} will allow us to derive a lower bound for the volume radius of our random polytopes $K_N$ that holds with overwhelming probability.

\subsection{Geometry of $\ell_p$-balls}

Recall that for $1 \leq p < \infty$ and $n\in\N$, we denote by $\ell_p^n$ the space $\R^n$ equipped with the norm $\|x\|_p = (\sum_{i=1}^n |x_i|^p)^{1/p}$. The unit ball of $\ell_p^n$ is denoted by $\B_p^n$ and we let $\SSS_p^{n-1}$ be the unit sphere in $\ell_p^n$. It is convenient for us to write $\B^n$ and $\SSS^{n-1}$ instead of $\B^n_2$ and $\SSS_2^{n-1}$, respectively. The volume of $\B_p^n$ is given by
\[
\vol_n(\B_p^n) = \frac{\big(2\Gamma(1+\frac{1}{p})\big)^n}{\Gamma(1+\frac{n}{p})}\,,
\]
see \cite[page 180]{AsymGeomAnalyBook}. It follows directly from Stirling's formula that asymptotically, as $n\to\infty$, $\vol_n(\B_p^n)^{1/n} \approx n^{-1/p}$. 

\smallskip

We rephrase the following result of Schechtman and Zinn \cite[Lemma 1]{SZ} (independently obtained by Rachev and R\"uschendorf in \cite{RR91}) that provides a probabilistic representation of the cone measure $\bM_{\B_p^n}$ of the unit ball of $\ell_p^n$ (see also \cite{BGMN} for an extension).

\begin{proposition}\label{thm:SZ}
Let $n\in\N$, $1\leq p < \infty$, and $g_1,\dots,g_n$ be independent real-valued random variables that are distributed according to the density
\[
f(t) = \frac{e^{-|t|^p}}{2\Gamma\big(1+{1\over p}\big)}, \qquad t\in\R\,.
\]
Consider the random vector $G=(g_1,\dots,g_n)\in\R^n$ and put $Y:=G/\|G\|_p$. Then $Y$ is independent of $\|G\|_p$ and has distribution $\bM_{\B^n_p}$.
\end{proposition}

Let us also recall the following result on moments of sums of independent random variables with log-concave tails that was originally obtained by Gluskin and Kwapie\'n in \cite[Theorem 1]{GK} and is stated here as in \cite[Proposition 7]{BGMN}, reflecting our particular set-up. In what follows, $\bE$ will denote expectation (integration) with respect to the probability measure $\bP$.

\begin{proposition}\label{thm:bgmn}
Let $n\in\N$, $1\leq p,q<\infty$ and $a=(a_i)_{i=1}^n\in\R^n$. Further, let the random vector $G$ be as in Proposition \ref{thm:SZ}. Then,
\begin{align}\label{eq:bgmn expression}
\big(\bE\,\big|\langle a,G \rangle\big|^q\big)^{1/q} \approx q^{1/p}\, \big\|(a_i^*)_{i=1}^q\big\|_{p^*} + \sqrt{q}\,\big\|(a_i^*)_{i=q+1}^n\big\|_2\,,
\end{align}
where $(a_i^*)_{i=1}^n$ is the non-increasing rearrangement of $(|a_i|)_{i=1}^n$ and $p^*$ is the conjugate of $p$ that is defined via the relation $\frac{1}{p}+\frac{1}{p^*}=1$.
\end{proposition}

\begin{remark}
The formulation of Proposition \ref{thm:bgmn} is taken from \cite{BGMN} and follows essentially from \cite[Theorem 1]{GK}, but requires computing and estimating the parameters that appear in \cite[Theorem 1]{GK}. The original formulation of the result contains a constant
$$
c_p=\bigg({1\over p}\bigg)^{1/p}\bigg({1\over p^*}\bigg)^{1/p^*}\,,
$$
see \cite[Remark 1]{GK}, and parameters $\kappa_1,\kappa_2$ depending on a convex function $N$ that determines the tail behaviour of the random variables involved. However, as can be seen in the proof of \cite[Proposition 7]{BGMN}, the equivalence in \eqref{eq:bgmn expression} holds with absolute constants. In the case that $q\geq n$, the right-hand side of \eqref{eq:bgmn expression} reduces to $q^{1/p}\|(a_i)_{i=1}^n \|_{p^*}$.
\end{remark}

\subsection{Two bounds for the average norm}
In order to bound the isotro\-pic constant of our random polytopes from above in the regime $2\leq p<\infty$, the following result is used. It was already applied in earlier papers concerning the isotropic constant of random polytopes and is explicitly written in \cite[Lemma 3.2]{ALT}.

\begin{proposition}\label{prop:IntegralBoundedBySup}
	Let $1\leq n \leq N$ and $P=\conv\{ P_1,\dots,P_N\}$ be a non-degenerated symmetric polytope in $\R^n$. Then,
	\[
	\frac{1}{\vol_n(P)}\int_P \|x\|_2^2\,\dint x \leq \frac{1}{(n+1)(n+2)}\sup_{\substack{E\subseteq \{1,\dots,N\}\\|E|=n}}\Bigg[ \sum_{i\in E}\|P_i\|_2^2 + \Big\| \sum_{i\in E} P_i \Big\|_2^2\Bigg]\,.
	\] 
\end{proposition}

For other closely related versions of Proposition \ref{prop:IntegralBoundedBySup} we refer the reader, for example, to \cite{A,IsotropicConvexBodies,KK}.

\medskip

Finally, in the case that $1\leq p<2$ we shall use instead the following bound for the integral of the $1$-norm. We directly formulate the result for the symmetric random polytopes $K_N$.

\begin{proposition}\label{prop:IntegralBoundL1Version}
Let $K_N$ be the symmetric random polytope generated by $n+1\leq N\leq e^{\sqrt{n}}$ independent random points $X_1,\ldots,X_N$ that are distributed on $\SSS_p^{n-1}$ according to ${\bf m}_{\mathbb{B}_p^n}$. Then,
\begin{align*}
{1\over\vol_n(K_N)}\int_{K_N}\|x\|_1\,\dint x \leq \max_{F=\conv(\{X_{i_1},\ldots,X_{i_n}\})\atop\varepsilon_{i_1},\ldots,\varepsilon_{i_n}=\pm 1}\!\!{1+\sqrt{2}\over n}\,\|\varepsilon_{i_1}X_{i_1}+\ldots+\varepsilon_{i_n}X_{i_n}\|_1\,,
\end{align*}
where the maximum is taken over all facets $F=\conv(\{X_{i_1},\ldots,X_{i_n}\})$ of $K_N$ with vertices $\{X_{i_1},\ldots,X_{i_n}\}\subseteq\{X_1,\ldots,X_N\}$ and over all choices of signs $\varepsilon_{i_1},\ldots,\varepsilon_{i_n}\in\{-1,+1\}$.
\end{proposition}
\begin{proof}
This is a direct consequence of \cite[Lemma 11.5.4]{IsotropicConvexBodies} and \cite[Identity (2.26)]{DGG} together with the fact that the $\ell_p$-balls $\B_p^n$ are symmetric with respect to all coordinate hyperplanes.
\end{proof}

\section{Proof of Theorem \ref{thm:Main}}\label{sec:ProofThm}

\subsection{The starting point}

Assume that $K\in\K^n$. It follows from the definition \eqref{eq:IsotropicConstantDef} of the isotropic constant $L_K$ of $K$ that
\[
nL_K^2 \leq \frac{1}{\vol_n(K)^{2/ n}}{1\over\vol_n(K)}\int_K \|x\|_2^2\,\dint x\,,
\]
and, as we stated in \eqref{eq:IsotropicConstantL1Bound},
\[
nL_K \leq {c\over\vol_n(K)^{1/ n}}{1\over\vol_n(K)}\int_K\|x\|_1\,\dint x
\]
for any symmetric $K$, $c\in(0,\infty)$ being an absolute constant.
Therefore, to prove boundedness of the isotropic constant of our random polytopes $K_N = \conv\{\pm X_1,\dots,\pm X_N \}$ with vertices $\pm X_1,\ldots,\pm X_N\in\SSS_p^{n-1}$ chosen with respect to the cone measure $\bM_{\B_p^n}$ of $\B_p^n$ with overwhelming probability, we proceed in two steps:
\begin{itemize}
\item Step 1: derive a high probability lower bound for $$\vol_n(K_N)^{1/n}\,.$$
\item Step 2: derive a high probability upper bound for $$\frac{1}{\vol_n(K_N)}\int_{K_N} \|x\|_2^2\,\dint x\,$$
if $2\leq p < \infty$, and for
\[
{1\over\vol_n(K_N)}\int_{K_N}\|x\|_1\,\dint x
\]
if $1\leq p < 2$.
\end{itemize}
Carrying out Step 1 and Step 2 is the content of the next three subsections. We subdivide Step 2 into the cases $1\leq p < 2$ (see Subsection \ref{sec:Step2p<2}) and $2\leq p <\infty$ (see Subsection \ref{sec:Step2}) and combine everything in Subsection \ref{sec:EndOfProof} to complete the proof. It will in fact become evident from the proof of the case $2\leq p <\infty$ that the case $1\leq p <2$ has to be treated differently.

This method of bounding the isotropic constant of random polytopes has been introduced in \cite{KK} and then been further developed in \cite{A,DGG}. As already explained earlier, we shall adapt this route to our model, which in particular means that in the regime $1\leq p<2$ we have to establish a new bound on the $\psi_2$-behaviour of the functional $\langle \,\cdot\,,\theta\rangle$, $\theta\in\R^n$, with respect to the cone measure on an $\ell_p$-sphere.

\subsection{Step 1}\label{sec:Step1}

We first introduce another random polytope model. We do this in a general framework in order to distinguish arguments working for general underlying convex bodies from those valid only for $\ell_p$-spheres. For that purpose, let $K\in\K^n$ be isotropic and symmetric and assume that $n+1\leq N\leq e^{\sqrt{n}}$. Now, let $Y_1,\ldots,Y_N$ be independent random points that are selected according to the uniform distribution on $K$, that is, the restriction of the Lebesgue measure to $K$, and define the random polytope
\[
\widetilde{K}_N=\conv\big(\{\pm Y_1,\ldots,\pm Y_N\}\big)\,.
\]
It should be noted that, in contrast to the random polytopes $K_N$, not all of the points $\pm Y_1,\ldots,\pm Y_N$ are necessarily vertices of $\widetilde{K}_N$. It is one of the main results of the paper \cite{DGT1} that
\begin{equation}\label{eq:RandomPolyCentroid}
\widetilde{K}_N \supseteq c_1 Z_{\log(2N/n)}(K)
\end{equation}
with probability at least $1-e^{-c_2\sqrt{N}}$, where $c_1,c_2\in(0,\infty)$ are absolute constants and $N\geq c_0n$ for another absolute constant $c_0\in(0,\infty)$. Together with \eqref{eq:volume centroid body} this immediately implies that
\begin{equation}\label{eq:volume random polytope}
\vol_n\big(\widetilde{K}_N\big)^{1/n}\gtrsim \sqrt{\frac{\log(2N/n)}{n}}\, L_K
\end{equation}
with probability bounded from below by $1-e^{-c_2\sqrt{N}}$, whenever $N\geq c_0 n$. If we work with a linear number of points, that is, with $n\leq N \leq \gamma n$ for some $\gamma>1$, then this goes back to Pivovarov \cite[Proposition 2]{Piv} (see also \cite[Theorem 11.3.7]{IsotropicConvexBodies} for a version combining both regimes).

In the remaining part of this section, let $K_N$ be a random polytope in $K$ that arises as the symmetric convex hull of $N$ random points distributed according to the cone probability measure $\bM_K$ on $K$. Our goal is to construct a coupling of these random polytopes $K_N$ and the random polytopes $\widetilde{K}_N$ introduced above such that pathwise the set inclusion $K_N\supseteq\widetilde{K}_N$ holds. To construct such a coupling, we let, as above, $Y_1,\ldots,Y_N$ be independent random points selected according to the uniform distribution on $K$ and let $(\Omega,\mathcal{F},\bP)$ be the underlying probability space on which these random points are defined. Now, we apply to each point $Y_i$ the Minkowski map $\mu_K$ and define $X_1(\omega):=\mu_K(Y_1(\omega)),\ldots,X_N(\omega):=\mu_K(Y_N(\omega))$ for all $\omega\in\Omega$ for which $\|Y_i\|_K\neq 0$ for all $i\in\{1,\ldots,N\}$ (in the other case, we put $X_1(\omega)=\ldots=X_N(\omega)$ to be some arbitrary boundary point of $K$). The random points $X_1,\ldots,X_N$ are clearly independent and are located on the boundary of $K$. Next, we notice that the push-forward measure of the uniform distribution on $K$ is -- by its very definition -- precisely the cone measure $\bM_K$ of $K$, i.e., for every Borel set $B\subseteq\bd\,K$ we have that
\begin{align*}
\bP(X_i\in B) &= \bP(\mu_K(Y_i)\in B) = \bP(Y_i\in\mu_K^{-1}(B))=\bM_K(B)\,.
\end{align*}
In other words this means that the independent random points $X_1,\ldots,X_N$ are distributed according to the cone measure $\bM_K$, implying that the random convex hull $\conv\big(\{\pm X_1,\ldots,\pm X_N\}\big)$ has the same distribution as our random polytope $K_N$. From now on we understand by $K_N$ the random polytope that arises in the way just described. Moreover, by construction, we have that $K_N(\omega)\supseteq\widetilde{K}_N(\omega)$ for all $\omega\in\Omega_0$, where $\Omega_0\subseteq\Omega$ is a subset satisfying $\bP(\Omega_0)=1$.  Using this coupling, we conclude from \eqref{eq:volume random polytope} that
\begin{equation}\label{eq:volume random polytopeFINAL}
\vol_n(K_N)^{1/n}\geq\vol_n(\widetilde{K}_N)^{1/n} \gtrsim\sqrt{\frac{\log(2N/n)}{n}}\, L_K
\end{equation}
on a subset of $\Omega$ with $\bP$-measure at least $1-e^{-c\sqrt{N}}$, where $c\in(0,\infty)$ is an absolute constant. This completes the proof of Step 1.

In the special case of the unit balls $\B_p^n$, we summarize the result in the following proposition. We emphasize that the additional factor $n^{-1/p}$ there just arises from the different volume normalization.

\begin{proposition}\label{prop:lower volume bound ell_p}
Let $1\leq p <\infty$ and $n+1\leq N \leq e^{\sqrt{n}}$. Assume that $X_1,\dots,X_N\in\SSS_p^{n-1}$ are independent random points with distribution $\bM_{\B_p^n}$ and write $K_N={\rm conv}(\{\pm X_1,\ldots,\pm X_N\})$. Then,
\[
\bP \Bigg( \vol_n(K_N)^{1/n} \geq c_1\frac{\sqrt{\log(2N/n)}}{n^{1/2+1/p}}\Bigg) \geq 1-e^{-c_2\sqrt{N}},
\] 
where $c_1,c_2\in(0,\infty)$ are absolute constants.
\end{proposition}

\begin{remark}
In principle, the coupling just described also works for $p=\infty$, in which case $\B_\infty^n$ is nothing else than the cube $[-1,1]^n$ in $\R^n$. However, since in the second step below we can only treat the case $p<\infty$, we decided to formulate Proposition \ref{prop:lower volume bound ell_p} merely in this set-up. 
\end{remark}


\subsection{Step 2 for $2\leq p<\infty$}\label{sec:Step2}

In view of the lower bound on the volume radius we obtained, to prove Theorem \ref{thm:Main} for $2\leq p<\infty$, it is enough to show that the inequality
 \[
 \frac{1}{\vol_n(K_N)}\int_{K_N} \|x\|_2^2\,\dint x \lesssim \frac{\log(2N/n)}{n^{2/p}}
 \]
holds with high probability. To prove this, we will apply Bernstein's inequality. Thus, we first have to make sure that the functionals $n^{1/p}\langle X,\theta\rangle$, where $X$ is distributed according to $\bM_{\B_p^n}$ and $\theta\in \SSS^{n-1}$, satisfy the required $\psi_2$-bound. Here we make use of the Schechtman-Zinn and Gluskin-Kwapie\'n results, see Propositions \ref{thm:SZ} and \ref{thm:bgmn}, respectively. We will show the following.

\begin{proposition}\label{thm:psi_2 bound}
There exists an absolute constant $c_{\psi_2}\in(0,\infty)$ such that
	\[
	\| n^{1/p}\langle\, \cdot\,,\theta \rangle \|_{\psi_2} \leq c_{\psi_2}
	\]
	for all $2\leq p<\infty$, $n\in\N$ and $\theta\in\SSS^{n-1}$, where $\|\,\cdot\,\|_{\psi_2}$ is the norm on the Orlicz space $L_{\psi_2}(\SSS_p^{n-1},\bM_{\B^n_p})$.
\end{proposition}

Before we turn to the proof of Proposition \ref{thm:psi_2 bound}, we first prove a lemma that provides an equivalent expression for the $q$th moment of $\langle X,\theta \rangle$ that will be more convenient to work with and that is also used in the next subsection.

\begin{lemma}\label{lem:qth moment estimate}
	For all $1\leq p,q<\infty$ and all $\theta=(\theta_1,\ldots,\theta_n)\in \R^n$, we have
	\begin{align*}
	&\bigg(\int_{\SSS_p^{n-1}} |\langle x,\theta \rangle|^q \,\dint\bM_{\B^n_p}(x)\bigg)^{1/q} \\
	&\qquad\approx \left(\frac{\Gamma(\frac{n}{p})}{\Gamma(\frac{n+q}{p})}\right)^{1/q}\Big[ q^{1/p}\, \big\|(\theta_i^*)_{i=1}^q\big\|_{p^*} + \sqrt{q}\,\big\|(\theta_i^*)_{i=q+1}^n\big\|_2 \Big]\,,
	\end{align*}
	where $(\theta_i^*)_{i=1}^n$ is the non-increasing rearrangement of $(\theta_i)_{i=1}^n$ and $p^*$ is determined by the equation ${1\over p}+{1\over p^*}=1$. 
\end{lemma}
\begin{proof}
	Let $\theta \in \SSS^{n-1}$. Then, by Proposition \ref{thm:SZ} and with $G$ as given there,
	\begin{align*}
	\bE \,\big|\big\langle G,\theta \big\rangle\big|^q & = \bE\,\bigg[\|G\|_p^q \cdot \Big|\Big\langle \frac{G}{\|G\|_p},\theta \Big\rangle \Big|^q\bigg] \cr
	& = \bE\, \|G\|_p^q \cdot \int_{\SSS_p^{n-1}}|\langle y,\theta \rangle|^q\,\dint\bM_{\B_p^n}(y).
	\end{align*} 	
	We now use the following polar integration formula to compute the $q$th moment of $\|G\|_p$. For $K\in\K^n$ one has that
	\begin{align*}
	\int_{\R^n}f(x)\,\dint x = n\,\vol_n(K)\int_0^\infty\int_{\bd\,K}f(ry)\,r^{n-1}\,\dint\bM_K(y)\dint r
	\end{align*}
	for all non-negative measurable functions $f:\R^n\to\R$ (in fact, this may alternatively be used as a definition for the cone measure $\bM_K$ of $K$).  We apply this to the function
	\[
	f(x) =  \frac{\|x\|_p^q\cdot e^{-\|x\|_p^p}}{\big(2\Gamma(1+1/p)\big)^n}\,,\qquad x\in\R^n,
	\]
	and the unit ball $\B^n_p$ of $\ell_p^n$, and obtain 
	\begin{equation*} 
	\begin{split}
	\bE\, \|G\|_p^q &=  \int_{\R^n} f(x) \, \dint x = \frac{n\,\vol_n(\B_p^n)}{(2\Gamma(1+\frac{1}{p}))^n} \, \int_0^\infty r^{n+q-1}e^{-r^p} \, \dint r  \\
	& = \frac{n\, \vol_n(\B_p^n) \, \Gamma(\frac{n+q}{p})}{p \, (2\Gamma(1+\frac{1}{p}))^n} = \frac{\Gamma(\frac{n+q}{p})}{\Gamma(\frac{n}{p})}  \,,
	\end{split}
	\end{equation*}
	where we also used the definition of the gamma function as well as the formula for $\vol_n(\B_p^n)$ stated in Section \ref{sec:prelim} above. 
	What is left to bound is the term $\bE\big[\big|\big\langle G,\theta \big\rangle\big|^q\big]$. This is done by means of Proposition \ref{thm:bgmn}, which yields that
	\[
	\bE \Big[\big|\big\langle G,\theta \big\rangle\big|^q\Big] \approx \Big[q^{1/p}\, \big\|(\theta_i^*)_{i=1}^q\big\|_{p^*} + \sqrt{q}\,\big\|(\theta_i^*)_{i=q+1}^n\big\|_2\Big]^q\,.
	\]
	This completes the proof.
\end{proof}

We can now prove the desired $\psi_2$-estimate for the linear functionals $\langle\,\cdot\,,\theta\rangle$, $\theta\in\SSS^{n-1}$, on the boundary of $\B_p^n$.

\begin{proof}[Proof of Proposition \ref{thm:psi_2 bound}]
	By definition of an Orlicz norm, it is enough to show that
	\[
	\int_{\SSS_p^{n-1}} \exp\bigg(\frac{n^{2/p}|\langle x,\theta \rangle|^2}{c_{\psi_2}^2}\bigg) \,\dint\bM_{\B^n_p}(x) \leq 2\,.
	\]
	Using the series expansion of the exponential function and Lemma \ref{lem:qth moment estimate} applied with $2q$ instead of $q$ there, we obtain that
	\begin{align*}
	&\int_{\SSS_p^{n-1}} \exp\bigg(\frac{n^{2/p}|\langle x,\theta \rangle|^2}{c_{\psi_2}^2}\bigg) \,\dint\bM_{\B^n_p}(x)\\ 
	& = \sum_{q=0}^\infty \frac{n^{2q/p}}{q!c_{\psi_2}^{2q}}\int_{\SSS_p^{n-1}} |\langle x,\theta \rangle|^{2q}\,\dint\bM_{\B^n_p}(x) \\
	& \leq 1+ \sum_{q=1}^\infty c_1 \frac{n^{2q/p}}{q!c_{\psi_2}^{2q}}\,\frac{\Gamma\big(\frac{n}{p}\big)}{\Gamma\big(\frac{n+2q}{p}\big)} \, \Big[ (2q)^{1/p}\, \big\|(\theta_i^*)_{i=1}^{2q}\big\|_{p^*} + \sqrt{2q}\,\big\|(\theta_i^*)_{i=2q+1}^n\big\|_2 \Big]^{2q}\,,
	\end{align*}
	where $(\theta_i^*)_{i=1}^n$ is the non-increasing rearrangement of $(\theta_i)_{i=1}^n$ and $c_1\in(0,\infty)$ is an absolute constant.
	Observe that, since $\theta\in \SSS^{n-1}$,
	$$
	\big\|(\theta_i^*)_{i=1}^{2q}\big\|_{p*}\leq(2q)^{{1\over p^*}-{1\over 2}}\big\|(\theta_i^*)_{i=1}^{2q}\big\|_2 \leq (2q)^{{1\over p^*}-{1\over 2}}
	$$
	if $1\leq p^*\leq 2$. In view of $1/p+1/p^*=1$, this implies that
		$$
		(2q)^{1/p}\, \big\|(\theta_i^*)_{i=1}^{2q}\big\|_{p^*} + \sqrt{2q}\,\big\|(\theta_i^*)_{i=2q+1}^n\big\|_2 \leq 2\sqrt{2q}\,,
		$$
		since $2\leq p<\infty$ has been assumed. Using Stirling's formula repeatedly, this yields
		\begin{align*}
		&\int_{\SSS_p^{n-1}} \exp\bigg(\frac{n^{2/p}|\langle x,\theta \rangle|^2}{c_{\psi_2}^2}\bigg) \,\dint\bM_{\B^n_p}(x)\\ 
		&\qquad\leq 1+c_1\sum_{q=1}^\infty {n^{2q/p}\over q!c_{\psi_2}^{2q}}\big(2\sqrt{2q}\big)^{2q}\frac{ \Gamma\big(\frac{n}{p}\big)}{\Gamma\big(\frac{n+2q}{p}\big)}\\
		&\qquad\leq 1+c_2\sum_{q=1}^\infty {n^{2q/p}\over q!c_{\psi_2}^{2q}}\big(2\sqrt{2q}\big)^{2q}\sqrt{\frac{n+2q}{n}}\Big(\frac{n}{n+2q}\Big)^{n/p}\Big(\frac{pe}{n+2q}\Big)^{2q/p}\\
		&\qquad \leq 1+c_2\sum_{q=1}^\infty {(pe)^{2q/p}\over q!c_{\psi_2}^{2q}}\big(2\sqrt{2q}\big)^{2q}2\sqrt{q}\\
		&\qquad \leq 1+c_3\sum_{q=1}^\infty {(2e)^{2q}\over c_{\psi_2}^{2q}}\big(2\sqrt{2q}\big)^{2q}\sqrt{q}\frac{1}{\sqrt{2\pi q}}\Big(\frac{e}{q}\Big)^q\\	
		&\qquad = 1+c_4\sum_{q=1}^\infty \bigg(\frac{32e^3}{c_{\psi_2}^2}\bigg)^q \,,			
		\end{align*}
		where $c_2,c_3,c_4\in(0,\infty)$ are again absolute constants. We can achieve that the last expression is bounded from above by $2$, provided that the constant $c_{\psi_2}\in(0,\infty)$ is chosen large enough. This completes the proof.
\end{proof}

Let us proceed with the proof of Step 2 by establishing the following result.

\begin{proposition}\label{prop:upperbound}
Let $2\leq p<\infty$ and $n+1\leq N \leq e^{\sqrt{n}}$. Then,
\begin{align*}
&\bP\Bigg(\frac{1}{\vol_n(K_N)}\int_{K_N} \|x\|_2^2\,\dint x \leq c_1\,\frac{\log(2N/n)}{n^{2/p}}\Bigg) \\
&\hspace{4cm}\geq 1-3\exp\Big(-c_2n\log(2N/n)\Big),
\end{align*}
where $c_1,c_2\in(0,\infty)$ are absolute constants.
\end{proposition}
\begin{proof}
Since $2\leq p<\infty$, the result of Proposition \ref{thm:psi_2 bound} ensures that we can apply Bernstein's inequality to the independent and centered random variables $n^{1/p}\langle X_i,\theta\rangle$, $i\in\{1,\ldots,N\}$, where $\theta\in\SSS^{n-1}$ is a fixed direction. Let $E\subseteq \{1,\dots,N\}$ with $|E|=n$. Bernstein's inequality (see Proposition \ref{prop:bernstein}) implies that
\begin{align*}
\bP\bigg(\bigg|\sum_{i\in E}\langle X_i,\theta\rangle\bigg|\geq \varepsilon n\bigg)\leq 2\exp\bigg(-{\varepsilon^2n^{2/p+1}\over 8c_{\psi_2}^2}\bigg)
\end{align*}
for all $\varepsilon>0$. Now, let $\mathcal N$ be a $1\over 2$-net of $\SSS^{n-1}$ with cardinality $|{\mathcal N}|\leq 5^n$ (the existence of such a net is ensured by \cite[Lemma 5.2.5]{AsymGeomAnalyBook}, for example). Then the union bound yields that
\begin{align*}
\bP\bigg(\exists\theta\in\mathcal{N}:\bigg|\sum_{i\in E}\langle X_i,\theta\rangle\bigg|\geq \varepsilon n\bigg) &\leq 2\cdot 5^n\exp\bigg(-{\varepsilon^2n^{2/p+1}\over 8c_{\psi_2}^2}\bigg)\\
&=2\exp\bigg(-{\varepsilon^2n^{2/p+1}\over 8c_{\psi_2}^2}+n\log 5\bigg)
\end{align*}
or, equivalently,
\begin{align*}
\bP\bigg(\forall\theta\in\mathcal{N}:\bigg|\sum_{i\in E}\langle X_i,\theta\rangle\bigg|< \varepsilon n\bigg) \geq 1-2\exp\bigg(-{\varepsilon^2n^{2/p+1}\over 8c_{\psi_2}^2}+n\log 5\bigg)\,.
\end{align*}

Now we want to show that the length of $\sum_{i\in E} X_i$ is bounded from above by $c\varepsilon n$ with overwhelming probability, where $c\in(0,\infty)$ is an absolute constant. To this end, note that each $\theta\in\SSS^{n-1}$ has a representation of the form $\theta = \sum_{j=1}^\infty \delta_j \theta_j$ with $\theta_j\in\mathcal N$ and $0 \leq \delta_j \leq (\frac{1}{2})^{j-1}$, see \cite{A}. Observe that for any $\omega\in\Omega$ satisfying $|\sum_{i\in E} \langle X_i(\omega),\theta\rangle| \leq \varepsilon n$ for every $\theta\in\mathcal N$, we have that 
\begin{align*}
\Big| \Big\langle \sum_{i\in E} X_i(\omega),\theta \Big\rangle \Big|
&= \Big| \Big\langle \sum_{i\in E} X_i(\omega),\sum_{j=1}^\infty \delta_j\theta_j \Big\rangle \Big| \\
&\leq \sum_{j=1}^\infty \delta_j \Big|\Big\langle \sum_{i\in E} X_i(\omega),\theta_j \Big\rangle \Big| \leq 2\varepsilon n
\end{align*}
for all $\theta\in\SSS^{n-1}$. Hence,
\begin{equation}\label{eq:small ball estimate euclidean norm of sum}
\begin{split}
\bP\bigg( \Big\|\sum_{i\in E} X_i \Big\|_2 \leq 2 \varepsilon n\bigg) 
& = \bP \bigg( \sup_{\theta\in \SSS^{n-1}}\Big\langle\sum_{i\in E} X_i, \theta\Big\rangle \leq 2 \varepsilon n\bigg) \\
& \geq 1- 2\exp\bigg(-{\varepsilon^2n^{2/p+1}\over 8c_{\psi_2}^2}+n\log 5\bigg)\,.
\end{split}
\end{equation}

Next, we notice that, according to Proposition \ref{prop:IntegralBoundedBySup},
\begin{align*}
& \bP\Bigg(\frac{1}{\vol_n(K_N)}\int_{K_N} \|x\|_2^2\,\dint x \leq c\frac{\log(2N/n)}{n^{2/p}}\Bigg) \\
& \geq \bP\Bigg(\sup_{\substack{E\subseteq \{1,\dots,N\}\\|E|=n}}\Bigg[ \sum_{i\in E}\|X_i\|_2^2 + \Big\| \sum_{i\in E} X_i \Big\|_2^2\Bigg] \leq cn^{2-2/p}\log(2N/ n)\Bigg)\,,
\end{align*}
where $c\in(0,\infty)$ is a sufficiently large absolute constant. Now, for every $x\in \R^n$, we have that $\| x\|_2 \leq n^{1/2-1/p} \|x\|_p$ by H\"older's inequality and the fact that $2\leq p<\infty$. Thus,
\begin{align*}
& \bP\Bigg(\sup_{\substack{E\subseteq \{1,\dots,N\}\\|E|=n}}\Bigg[ \sum_{i\in E}\|X_i\|_2^2 + \Big\| \sum_{i\in E} X_i \Big\|_2^2\Bigg] \leq cn^{2-2/p}\log(2N/ n)\Bigg) \\
& \geq \bP\Bigg(\sup_{\substack{E\subseteq \{1,\dots,N\}\\|E|=n}} \Big\| \sum_{i\in E} X_i \Big\|_2^2 \leq cn^{2-2/p}\log(2N/ n)-n^{2-2/p}\Bigg) \\
& \geq \bP\Bigg(\sup_{\substack{E\subseteq \{1,\dots,N\}\\|E|=n}} \Big\| \sum_{i\in E} X_i \Big\|_2^2 \leq c_1n^{2-2/p}\log(2N/ n)\Bigg) \\
& = \bP\Bigg(\sup_{\substack{E\subseteq \{1,\dots,N\}\\|E|=n}} \Big\| \sum_{i\in E} X_i \Big\|_2 \leq \sqrt{c_1}n^{1-1/p}\sqrt{\log(2N/ n)}\Bigg)\,,
\end{align*}
where $c_1\in(0,\infty)$ is an absolute constant. Therefore, changing to the complementary event, using the union bound and applying \eqref{eq:small ball estimate euclidean norm of sum} with $\varepsilon=\sqrt{c_1}n^{-1/p}\sqrt{\log(2N/n)}/2$, together with Stirling's formula, we obtain
\begin{align*}
& \bP\Bigg(\frac{1}{\vol_n(K_N)}\int_{K_N} \|x\|_2^2\,\dint x \leq c\frac{\log(2N/n)}{n^{2/p}}\Bigg) \\
& \geq 1-2{{N}\choose{n}}\exp\Bigg(-\frac{c_1}{32c_{\psi_2}^2}n\log(2N/n)+n \log 5\Bigg) \\
& \geq 1-2\exp\Bigg(-\frac{c_1}{32c_{\psi_2}^2}n\log(2N/n)+ n\log(2eN/n) +n \log 5\Bigg) \\
& \geq 1-2 \exp\Big(-c_2n\log(2N/n)\Big)\,,
\end{align*}
where we choose the constant $c$ above so large that $c_1/(32c_{\psi_2}^2)\geq 1$ and put $c_2:=(1+2+\log 5)c_1/(32c_{\psi_2}^2)=c_1/(8c_{\psi_2}^2)$, for example.
\end{proof}

\subsection{Step 2 for $1\leq p<2$}\label{sec:Step2p<2}

The purpose of this section is to establish a high-probability upper bound for the expression
$$
{1\over\vol_n(K_N)}\int_{K_N}\|x\|_1\,\dint x
$$
and to prove Step 2 in the case that $K_N$ is a symmetric random polytope with vertices from $\SSS_p^{n-1}$ with $1\leq p<2$. To this end, we first establish a suitable substitute for Proposition \ref{thm:psi_2 bound}, which can be regarded as the analogue for the cone measure of $\B_p^n$ of a theorem of Bobkov and Nazarov \cite{BobkovNazarov}. In what follows, we shall write $\SSS_\infty^{n-1}$ for the unit sphere in $\R^n$ supplied with the $\ell_\infty$-norm $\|x\|_\infty=\max\{|x_i|:i=1,\ldots,n\}$, $x=(x_1,\ldots,x_n)\in\R^n$.

\begin{proposition}\label{prop:Psi_2boundForp<2}
There exists an absolute constant $c_{\psi_2}\in(0,\infty)$ such that, for every $1\leq p < 2$ and all $\theta\in\SSS_{\infty}^{n-1}$,
\[
\Big\|{1\over n^{1/2-1/p}}\langle\,\cdot\,,\theta \rangle\Big\|_{\psi_2} \leq c_{\psi_2}\,,
\]
where $\| \cdot\|_{\psi_2}$ is the norm on the Orlicz space $L_{\psi_2}(\SSS_p^{n-1},\bM_{\B_p^n})$.
\end{proposition}
\begin{proof}
Fix $1\leq p < 2$ and $\theta=(\theta_1,\dots,\theta_n)\in\SSS_{\infty}^{n-1}$. By Proposition \ref{prop:psi 2 norm equivalence}, we have that
\[
\|\langle\,\cdot\,,\theta \rangle\|_{\psi_2} \lesssim \sup_{q\geq 2}\frac{\|\langle\,\cdot\,,\theta \rangle\|_{q}}{\sqrt{q}}\,,
\]
where $\|\,\cdot\,\|_{q}$ is the $L_q$-norm on $\SSS_p^{n-1}$. Let us now consider the two cases $q<n$ and $q\geq n$ separately. From Lemma \ref{lem:qth moment estimate} we obtain that, for any $1\leq q <\infty$, 
\begin{align*}
\|\langle\,\cdot\,,\theta \rangle\|_{q} & \lesssim \left(\frac{\Gamma(\frac{n}{p})}{\Gamma(\frac{n+q}{p})}\right)^{1/q}\Big[ q^{1/p}\, \big\|(\theta_i^*)_{i=1}^q\big\|_{p^*} + \sqrt{q}\,\big\|(\theta_i^*)_{i=q+1}^n\big\|_2 \Big]\cr
& \lesssim
\begin{cases}
\big[q+\sqrt{q}\,\sqrt{n-q}\,\big]\left(\frac{\Gamma(\frac{n}{p})}{\Gamma(\frac{n+q}{p})}\right)^{1/q} &: 2\leq q< n \\
q \left(\frac{\Gamma(\frac{n}{p})}{\Gamma(\frac{n+q}{p})}\right)^{1/q}&: q\geq n\,.
\end{cases} 
\end{align*} 
Here, we have estimated in the second step the two norms against the $\ell_\infty$-norm and used the fact that ${1\over p}+{1\over p^*}=1$ and that $\|\theta\|_\infty=1$. Applying Stirling's formula, we see that
\begin{align*}
\left(\frac{\Gamma(\frac{n}{p})}{\Gamma(\frac{n+q}{p})}\right)^{1/q} &\lesssim
 n^{(\frac{n}{p}-\frac{1}{2})\frac{1}{q}}(n+q)^{(\frac{1}{2}-\frac{n+q}{p})\frac{1}{q}} \cr
& \lesssim (n+q)^{(\frac{n}{p}-\frac{1}{2})\frac{1}{q}}(n+q)^{(\frac{1}{2}-\frac{n+q}{p})\frac{1}{q}} =  \frac{1}{(n+q)^{1/p}}\,.
\end{align*}
Now, we note that if $2\leq q<n$, $\frac{\|\langle\,\cdot\,,\theta \rangle\|_{q}}{\sqrt{q}}\lesssim{\sqrt{q}+\sqrt{n-q}\over(n+q)^{1/p}}$ is maximal for $q=2$, where it is $c_1n^{1/2-1/p}$ with some absolute constant $c_1\in(0,\infty)$. Similarly, if $q\geq n$, $\frac{\|\langle\,\cdot\,,\theta \rangle\|_{q}}{\sqrt{q}}\lesssim{\sqrt{q}\over (n+q)^{1/p}}$ is maximal for the choice $q=n$, where the expression reduces to $c_2n^{1/2-1/p}$ with another absolute constant $c_2\in(0,\infty)$. As a consequence, we find that
$$
\|\langle\,\cdot\,,\theta \rangle\|_{\psi_2} \lesssim n^{1/2-1/p}
$$
and the proof is complete.
\end{proof}

We can now proceed and derive the announced high-probability upper bound by partially following the strategy in \cite{DGG} (see also \cite[Chapter 11.5.1]{IsotropicConvexBodies}).

\begin{proposition}
Let $1\leq p<2$ and $n+1\leq N\leq e^{\sqrt{n}}$. Then,
\begin{align*}
&\bP\Bigg(\frac{1}{\vol_n(K_N)}\int_{K_N} \|x\|_1\,\dint x \leq c_1\,n^{{1\over 2}-{1\over p}}\sqrt{\log\Big({2N\over n}}\Big)\Bigg) \\
&\hspace{4cm}\geq 1-2\exp\Big(-c_2n\log(2N/n)\Big),
\end{align*}
where $c_1,c_2\in(0,\infty)$ are absolute constants.
\end{proposition}
\begin{proof}
Let $X_1,\ldots,X_n$ be independent and identically distributed on $\SSS_p^{n-1}$ according to the cone measure of $\B_p^n$, and fix $\theta\in\SSS_\infty^{n-1}$ as well as $\varepsilon_1,\ldots,\varepsilon_n\in\{-1,+1\}$. We define $g_j(X_1,\ldots,X_n):=\langle\varepsilon_jX_j,\theta\rangle$ for $j\leq n$ and notice that $\|g_j\|_{\psi_2}\lesssim n^{1/2-1/p}$ according to Proposition \ref{prop:Psi_2boundForp<2} and since $\B_p^n$ is symmetric with respect to all coordinate hyperplanes. Now, using Bernstein's inequality (see Proposition \ref{prop:bernstein}), we conclude that
\begin{align*}
\bP\bigg(\Big|\sum_{j=1}^ng_j(X_1,\ldots,X_n)\Big|>t\,n\bigg) &=\bP\big(|\langle\varepsilon_1X_1+\ldots+\varepsilon_nX_n,\theta\rangle|>t\,n\big)\\
&\leq 2\exp\big(-c\,t^2\,n^{2/p}\big)
\end{align*}
for all $t>0$ and with an absolute constant $c\in(0,\infty)$.
Thus, choosing $t=c_1n^{1/2-1/p}\sqrt{\log\big({2N\over n}\big)}$ with some absolute constant $c_1\in(0,\infty)$, we have that
\begin{equation}\label{eq:ProofL1Gleichung1}
\begin{split}
&\bP\Bigg(|\langle\varepsilon_1X_1+\ldots+\varepsilon_nX_n,\theta\rangle|>c_1n^{{3\over 2}-{1\over p}}\sqrt{\log\Big({2N\over n}\Big)}\,\Bigg)\\
&\qquad\qquad\qquad\leq 2\exp\Big(-c_2n\log\Big({2N\over n}\Big)\Big)
\end{split}
\end{equation}
with $c_2=c\cdot c_1$. Next, we notice that
$$
\sup_{\theta\in\SSS_\infty^{n-1}}|\langle\varepsilon_1X_1+\ldots+\varepsilon_nX_n,\theta\rangle| = \sup_{\theta\in{\rm extr}(\SSS_\infty^{n-1})}|\langle\varepsilon_1X_1+\ldots+\varepsilon_nX_n,\theta\rangle|\,,
$$
where ${\rm extr}(\SSS_\infty^{n-1})$ is the set of extreme points of $\SSS_\infty^{n-1}$. Clearly, any such extreme point has the representation $\theta=(\theta_1,\ldots,\theta_n)$ with $\theta_1,\ldots,\theta_n\in\{-1,+1\}$. Thus, using \eqref{eq:ProofL1Gleichung1} and the union bound, we find that
\begin{align*}
&\bP\Bigg(\max_{\varepsilon_j=\pm 1}\|\varepsilon_1X_1+\ldots+\varepsilon_nX_n\|_1 > c_3n^{{3\over 2}-{1\over p}}\sqrt{\log\Big({2N\over n}\Big)}\Bigg)\\
&=\bP\Bigg(\max_{\varepsilon_j=\pm 1}\sup_{\theta\in\SSS_\infty^{n-1}}|\langle\varepsilon_1X_1+\ldots+\varepsilon_nX_n,\theta\rangle| > c_3n^{{3\over 2}-{1\over p}}\sqrt{\log\Big({2N\over n}\Big)}\Bigg)\\
&=\bP\Bigg(\max_{\varepsilon_j=\pm 1}\sup_{\theta\in\{-1,+1\}^n}|\langle\varepsilon_1X_1+\ldots+\varepsilon_nX_n,\theta\rangle| > c_3n^{{3\over 2}-{1\over p}}\sqrt{\log\Big({2N\over n}\Big)}\Bigg)\\
&\leq\sum_{\varepsilon_1,\ldots,\varepsilon_n\in\{-1,+1\}\atop \theta\in\{-1,+1\}^n}\bP\Bigg(|\langle\varepsilon_1X_1+\ldots+\varepsilon_nX_n,\theta\rangle| > c_3n^{{3\over 2}-{1\over p}}\sqrt{\log\Big({2N\over n}\Big)}\Bigg)\\
&\leq 2^n\cdot 2^n\cdot 2\exp\Big(-c_2n\log\Big({2N\over n}\Big)\Big)\\
&= 2\exp\Big(-c_2n\log\Big({2N\over n}\Big)+n\log 4\Big)\\
&\leq 2\exp\Big(-c_4n\log\Big({2N\over n}\Big)\Big)
\end{align*}
with an absolute constant $c_4\in(0,\infty)$, whenever the constant $c_1$ above is chosen sufficiently large. Equivalently, this can be re-phrased by saying that
\begin{align*}
&\bP\Bigg(\max_{\varepsilon_j=\pm 1}\|\varepsilon_1X_1+\ldots+\varepsilon_nX_n\|_1 \leq c_3n^{{3\over 2}-{1\over p}}\sqrt{\log\Big({2N\over n}\Big)}\Bigg)\\
&\qquad\qquad\qquad\geq 1-2\exp\Big(-c_4n\log\Big({2N\over n}\Big)\Big)\,.
\end{align*}
The result follows now from Proposition \ref{prop:IntegralBoundL1Version}, since each facet of $K_N$ has the same distribution as the convex hull $\conv(\{X_1,\ldots,X_n\})$ of $X_1,\ldots,X_n$.
\end{proof}

\subsection{Completion of the proof}\label{sec:EndOfProof}
Using the estimates we obtained in Step 1 and Step 2, we will now complete the proof of Theorem \ref{thm:Main}. Recall that for every $K\in\K^n$,
\[
nL_K^2 \leq \frac{1}{\vol_n(K)^{2/ n}}{1\over\vol_n(K)}\int_K \|x\|_2^2\,\dint x\,.
\]
Thus, by combining this with Proposition \ref{prop:lower volume bound ell_p} and Proposition \ref{prop:upperbound}, we can find absolute constants $c_1,c_2,c_3\in(0,\infty)$ such that for all $2\leq p<\infty$, $n+1\leq N\leq e^{\sqrt{n}}$ and all random polytopes $K_N = \conv\{\pm X_1,\dots,\pm X_N \}$ with vertices $\pm X_1,\ldots,\pm X_N$ chosen with respect to the cone measure $\bM_{\B_p^n}$ from $\SSS_p^{n-1}$, 
\begin{align*}
\bP\big(L_{K_N} \leq c_1\big) & \geq \bP\left(\frac{1}{\vol_n(K_N)^{2/ n}}{1\over\vol_n(K_N)}\int_{K_N} \|x\|_2^2\,\dint x \leq n\,c_1^2 \right)\\
& \geq 1- e^{-c_2\sqrt{N}} - 2e^{-c_3n\log(2N/n)}\,.
\end{align*}
Now, recall that
\[
nL_K \leq {c\over\vol_n(K)^{1+{1/ n}}}\int_K\|x\|_1\,\dint x
\]
for any symmetric $K$, where $c\in(0,\infty)$ is an absolute constant. Hence, in the situation $1\leq p<2$, combining Proposition \ref{prop:lower volume bound ell_p} with Proposition \ref{prop:Psi_2boundForp<2}, we conclude that there are absolute constants $c_4,c_5,c_6\in(0,\infty)$ such that for all $1\leq p<2$, $n+1\leq N\leq e^{\sqrt{n}}$ and all random polytopes $K_N = \conv\{\pm X_1,\dots,\pm X_N \}$ with vertices $\pm X_1,\ldots,\pm X_N$ chosen with respect to the cone measure $\bM_{\B_p^n}$ from $\SSS_p^{n-1}$, 
\begin{align*}
\bP\big(L_{K_N} \leq c_4\big) & \geq \bP\left(\frac{1}{\vol_n(K_N)^{1/ n}}{1\over\vol_n(K_N)}\int_{K_N} \|x\|_1\,\dint x \leq n\,c_4 \right)\\
& \geq 1- e^{-c_5\sqrt{N}} - 2e^{-c_6n\log(2N/n)}\,.
\end{align*}
This completes the proof of Theorem \ref{thm:Main}.\hfill $\Box$

\section*{Acknowledgement}
We would like to thank David Alonso-Guti\'errez and Apostolos Giannopoulos for useful conversations and interesting hints and remarks. We would also like to thank an anonymous referee for many helpful suggestions and especially for pointing us to an error in an earlier version of this manuscript. The financial support of the Mercator Research Center Ruhr has made possible a research stay of the second author at Ruhr University Bochum.

\bibliographystyle{plain}
\bibliography{random_polytopes}

\def\cprime{$'$} \def\cprime{$'$}
\begin{thebibliography}{10}

\bibitem{A}
D.~Alonso-Guti{\'e}rrez.
\newblock On the isotropy constant of random convex sets.
\newblock {\em Proc. Amer. Math. Soc.}, 136(9):3293--3300, 2008.

\bibitem{A11}
D.~Alonso-Guti{\'e}rrez.
\newblock A remark on the isotropy constant of polytopes.
\newblock {\em Proc. Amer. Math. Soc.}, 139(7):2565--2569, 2011.

\bibitem{ALT}
D.~Alonso-Guti{\'e}rrez, A.~E. Litvak, and N.~Tomczak-Jaegermann.
\newblock On the isotropic constant of random polytopes.
\newblock {\em J. Geom. Anal.}, 26(1):645--662, 2016.

\bibitem{ABBW}
D.~Alonso-GutiÃ©rrez, J.~Bastero, J.~BernuÃ©s, and P.~Wolff.
\newblock On the isotropy constant of projections of polytopes.
\newblock {\em J. Funct. Anal.}, 258(5):1452--1465, 2010.

\bibitem{AsymGeomAnalyBook}
S.~Artstein-Avidan, A.~Giannopoulos, and V.~D. Milman.
\newblock {\em Asymptotic {G}eometric {A}nalysis. {P}art {I}}, volume 202 of
  {\em Mathematical Surveys and Monographs}.
\newblock American Mathematical Society, Providence, RI, 2015.

\bibitem{B}
K.~Ball.
\newblock Normed spaces with a weak-{G}ordon-{L}ewis property.
\newblock In {\em Functional analysis ({A}ustin, {TX}, 1987/1989)}, volume 1470
  of {\em Lecture Notes in Math.}, pages 36--47. Springer, Berlin, 1991.

\bibitem{BGMN}
F.~Barthe, O.~Gu{\'e}don, S.~Mendelson, and A.~Naor.
\newblock A probabilistic approach to the geometry of the {$l^n_p$}-ball.
\newblock {\em Ann. Probab.}, 33(2):480--513, 2005.

\bibitem{BobkovNazarov}
S.~G. Bobkov and F.~L. Nazarov.
\newblock Large deviations of typical linear functionals on a convex body with
  unconditional basis.
\newblock In {\em Stochastic inequalities and applications}, volume~56 of {\em
  Progr. Probab.}, pages 3--13. Birkh\"auser, Basel, 2003.

\bibitem{BoeroeEtAl}
K.~J. B{\"o}r{\"o}czky, F.~Fodor, and D.~Hug.
\newblock Intrinsic volumes of random polytopes with vertices on the boundary
  of a convex body.
\newblock {\em Trans. Amer. Math. Soc.}, 365(2):785--809, 2013.

\bibitem{Bou}
J.~Bourgain.
\newblock On high-dimensional maximal functions associated to convex bodies.
\newblock {\em Amer. J. Math.}, 108(6):1467--1476, 1986.

\bibitem{B91}
J.~Bourgain.
\newblock On the distribution of polynomials on high dimensional convex sets.
\newblock In {\em Geometric aspects of functional analysis}, volume 1469 of
  {\em Lecture Notes in Math.}, pages 127--137. Springer, Berlin, 1991.

\bibitem{BLM88}
J.~Bourgain, J.~Lindenstrauss, and V.D. Milman.
\newblock Minkowski sums and symmetrizations.
\newblock In {\em Geometric aspects of functional analysis}, volume 1317 of
  {\em Lecture Notes in Math.}, pages 44--74. Springer, Berlin, 1988.

\bibitem{IsotropicConvexBodies}
S.~Brazitikos, A.~Giannopoulos, P.~Valettas, and B.-H. Vritsiou.
\newblock {\em Geometry of {I}sotropic {C}onvex {B}odies}, volume 196 of {\em
  Mathematical Surveys and Monographs}.
\newblock American Mathematical Society, Providence, RI, 2014.

\bibitem{DGG}
N.~Dafnis, A.~Giannopoulos, and O.~Gu{\'e}don.
\newblock On the isotropic constant of random polytopes.
\newblock {\em Adv. Geom.}, 10(2):311--322, 2010.

\bibitem{DGT1}
N.~Dafnis, A.~Giannopoulos, and A.~Tsolomitis.
\newblock Asymptotic shape of a random polytope in a convex body.
\newblock {\em J. Funct. Anal.}, 257(9):2820--2839, 2009.

\bibitem{GHT}
A.~Giannopoulos, L.~Hioni, and A.~Tsolomitis.
\newblock Asymptotic shape of the convex hull of isotropic log-concave random
  vectors.
\newblock {\em Adv. Appl. Math.}, 75:116--143, 2016.

\bibitem{G}
E.~D. Gluskin.
\newblock The diameter of the {M}inkowski compactum is roughly equal to {$n$}.
\newblock {\em Funktsional. Anal. i Prilozhen.}, 15(1):72--73, 1981.

\bibitem{GK}
E.~D. Gluskin and S.~Kwapie\'n.
\newblock Tail and moment estimates for sums of independent random variables
  with logarithmically concave tails.
\newblock {\em Studia Math.}, 114(3):303--309, 1995.

\bibitem{He}
D.~Hensley.
\newblock Slicing convex bodies---bounds for slice area in terms of the body's
  covariance.
\newblock {\em Proc. Amer. Math. Soc.}, 79(4):619--625, 1980.

\bibitem{HH14}
J.~H\"orrmann and D.~Hug.
\newblock On the volume of the zero cell of a class of isotropic poisson
  hyperplane tessellations.
\newblock {\em Adv. Appl. Probab.}, 46:622--642, 2014.

\bibitem{HHRT}
J.~H{\"o}rrmann, D.~Hug, M.~Reitzner, and C.~Th{\"a}le.
\newblock Poisson polyhedra in high dimensions.
\newblock {\em Adv. Math.}, 281:1--39, 2015.

\bibitem{J94}
M.~Junge.
\newblock Hyperplane conjecture for quotient spaces of lp.
\newblock {\em Forum Math.}, 6(5):617--636, 1994.

\bibitem{K06}
B.~Klartag.
\newblock On convex perturbations with a bounded isotropic constant.
\newblock {\em Geom. Funct. Anal.}, 16(6):1274--1290, 2006.

\bibitem{KK}
B.~Klartag and G.~Kozma.
\newblock On the hyperplane conjecture for random convex sets.
\newblock {\em Israel J. Math.}, 170:253--268, 2009.

\bibitem{KM12}
B.~Klartag and E.~Milman.
\newblock Centroid bodies and the logarithmic {L}aplace transform -- {A}
  unified approach.
\newblock {\em J. Funct. Anal.}, 262(1):10--34, 2012.

\bibitem{KMP}
H.~K{\"o}nig, M.~Meyer, and A.~Pajor.
\newblock The isotropy constants of the {S}chatten classes are bounded.
\newblock {\em Math. Ann.}, 312(4):773--783, 1998.

\bibitem{MP}
V.~D. Milman and A.~Pajor.
\newblock Isotropic position and inertia ellipsoids and zonoids of the unit
  ball of a normed {$n$}-dimensional space.
\newblock In {\em Geometric aspects of functional analysis (1987--88)}, volume
  1376 of {\em Lecture Notes in Math.}, pages 64--104. Springer, Berlin, 1989.

\bibitem{N}
A.~Naor.
\newblock The surface measure and cone measure on the sphere of $\ell_p^n$.
\newblock {\em Trans. Amer. Math. Soc.}, 359(3):1045--1079, 2007.

\bibitem{NR2002}
A.~Naor and D.~Romik.
\newblock Projecting the surface measure of the sphere of {$\ell_p^n$}.
\newblock {\em Ann. Inst. H. Poincar\'e Probab. Statist.}, 39(2):241--261,
  2003.

\bibitem{Paouris2}
G.~Paouris.
\newblock Concentration of mass and central limit properties of isotropic
  convex bodies.
\newblock {\em Proc. Amer. Math. Soc.}, 133(2):565--575, 2005.

\bibitem{Paouris}
G.~Paouris.
\newblock On the $\psi_2 $-behaviour of linear functionals on isotropic convex
  bodies.
\newblock {\em Studia. Math.}, 168:285--299, 2005.

\bibitem{P06}
G.~Paouris.
\newblock Concentration of mass on convex bodies.
\newblock {\em Geom. Funct. Anal.}, 16(5):1021--1049, 2006.

\bibitem{Piv}
P.~Pivovarov.
\newblock On determinants and the volume of random polytopes in isotropic
  convex bodies.
\newblock {\em Geometriae Dedicata}, 149(1):45--58, 2010.

\bibitem{RR91}
S.~T. Rachev and L.~R\"uschendorf.
\newblock Approximate independence of distributions on spheres and their
  stability properties.
\newblock {\em Ann. Probab.}, 19(3):1311--1337, 1991.

\bibitem{Reitzner2002}
M.~Reitzner.
\newblock Random points on the boundary of smooth convex bodies.
\newblock {\em Trans. Amer. Math. Soc.}, 354(6):2243--2278 (electronic), 2002.

\bibitem{Reitzner2003}
M.~Reitzner.
\newblock Random polytopes and the {E}fron-{S}tein jackknife inequality.
\newblock {\em Ann. Probab.}, 31(4):2136--2166, 2003.

\bibitem{RichardsonVuWu}
R.~M. Richardson, V.~H. Vu, and L.~Wu.
\newblock An inscribing model for random polytopes.
\newblock {\em Discrete Comput. Geom.}, 39(1-3):469--499, 2008.

\bibitem{SZ}
G.~Schechtman and J.~Zinn.
\newblock On the volume of the intersection of two {$L^n_p$} balls.
\newblock {\em Proc. Amer. Math. Soc.}, 110(1):217--224, 1990.

\bibitem{SZ2}
G.~Schechtman and J.~Zinn.
\newblock {\em Geometric Aspects of Functional Analysis: Israel Seminar
  1996--2000}, chapter Concentration on the {$\ell_p^n$} ball, pages 245--256.
\newblock Springer Berlin Heidelberg, Berlin, Heidelberg, 2000.

\bibitem{SchuettWernerRPBoundary}
C.~Sch{\"u}tt and E.~Werner.
\newblock Polytopes with vertices chosen randomly from the boundary of a convex
  body.
\newblock In {\em Geometric aspects of functional analysis}, volume 1807 of
  {\em Lecture Notes in Math.}, pages 241--422. Springer, Berlin, 2003.

\end{thebibliography}

\end{document}